\documentclass[11pt]{article}

\usepackage{latexsym}
\usepackage{bbm}
\usepackage{amsmath}
\usepackage{amsfonts}
\usepackage{amssymb}
\usepackage{multirow}
\usepackage{latexsym}
\usepackage[dvips]{graphicx}
\usepackage{overpic}
\usepackage{graphicx}
\usepackage{relsize}
\usepackage{leftidx}
\allowdisplaybreaks
\newtheorem{theorem}{\bf Theorem}[section]
\newtheorem{lemma}[theorem]{\bf Lemma}

\newtheorem{defn}{\bf Definition}[section]

\newtheorem{remark}{{\bf Remark}}[section]
\newenvironment{proof}{\noindent{\em Proof.}}{\quad \hfill$\Box$\vspace{2ex}}

\def \bR {\Bbb R}

\def \and {\, \mb
ox{\rm and}\, }

\def \supp {\,{\rm supp}\,}

\def \Re {\,{\rm Re}\,}

\def \min{\,{\rm min}}

\def \l {\left}
\def \r {\right}
\makeatletter

\newcommand{\Rmnum}[1]{\expandafter\@slowromancap\romannumeral #1@}

\makeatother
\begin{document}

\title {\bf Sharp Bounds for Oscillatory Integral Operators with Homogeneous Polynomial Phases}

\author{Danqing He \thanks{Department of Mathematics, Sun Yat-sen (Zhongshan) University, Guangzhou, 510275, P.R. China. E-mail address: {\it hedanqing@mail.sysu.edu.cn}.}\,\,\,,\quad
Zuoshunhua Shi \thanks{School of Mathematics and Statistics, Central South University, Changsha, People's Republic of China. E-mail address:
        {\it shizsh@csu.edu.cn}.} }

\date{}
\maketitle{}

\begin{abstract}
We obtain sharp $L^p$ bounds for oscillatory integral operators with generic homogeneous polynomial phases in several variables. The phases considered in this paper satisfy the rank one condition which is an important notion introduced by Greenleaf, Pramanik and Tang. Under certain additional assumptions, we can establish sharp damping estimates with critical exponents to prove endpoint $L^p$ estimates.
\end{abstract}


\section{Introduction}
Let $T_{\lambda}$ be an oscillatory integral operator of the form
\begin{equation}\label{sec1 general oio}
T_{\lambda}f(x)=\int_{\bR^{n_Y}} e^{i\lambda S(x,y)}
\varphi(x,y)f(y)dy, ~~~x\in \bR^{n_X},
\end{equation}
where $n_X,n_Y$ are two positive integers, $\lambda$ is a real parameter, $S$ is a real-valued smooth function and $\varphi$ is a smooth cut-off function near the origin in $\bR^{n_X}\times\bR^{n_Y}$. We shall refer to $T_{\lambda}$ as an $(n_X+n_Y)-$dimensional oscillatory integral operator. In this paper, our purpose is to establish sharp $L^p$ bounds for these operators with homogeneous polynomial phases.

In the $(1+1)-$dimensional setting, Phong and Stein \cite{PS1997} proved the remarkable theorem that the sharp $L^2$ decay estimate of $T_{\lambda}$ is determined by the Newton polyhedron of the real -analytic phase $S$; for some related results on oscillatory integrals and oscillatory integral operators see \cite{varchenko,hormander2,PS1992,PS1994,seeger2,
greenblatt1,Gilula}. This result was extended to the case of smooth phases by Greenblatt \cite{greenblatt2}; see Rychkov \cite{rychkov} for a partial result. On the other hand, estimates of $T_{\lambda}$ on $L^p$ were also studied by many authors \cite{greenleafseeger1,PSS2001,yangchanwoo,
yangchanwoo2,gressmanxiao,ShiYan,Shi,ShiXuYan,Gilula-Gressman-Xiao}. Recently, general sharp $L^p$ decay estimates have been proved by Xiao \cite{Xiao2017}. For a survey on degenerate oscillatory and Fourier integral operators, we refer the reader to Greenleaf-Seeger \cite{greenleafseeger2}.

It is difficult to generalize all of the one dimensional results to general higher dimensional cases. However, some uniform estimates were also obtained with non-sharp decay rates; see \cite{carbery,CaW,Christ-Li-Tao-Thiele,PSS2001}.
Greenleaf, Pramanik and Tang \cite{GPT} introduced the important notion of rank one condition to establish sharp $L^2$ estimates when the phases are generic homogeneous polynomials; for a earlier result in $(2+1)$ dimensions, see \cite{tangwan}. Under the rank one condition, we shall extend the $L^2$ result in Greenleaf-Pramanik-Tang \cite{GPT} to the $L^p$ setting in this paper. In this direction, Xu and Yan \cite{xuyan} considered the special case $n_X=n_Y$ and obtained sharp $L^p$ estimates for $T_{\lambda}$.

In Section 2, we will prepare some basic tools for our argument. Sharp $L^p$ estimates and endpoint estimates will be given in Sections 3 and 4. For two real numbers $a$ and $b$, we use $a\lesssim b$ to mean $a\leq Cb$ for some constant $C>0$, and $a\approx b$ to mean $a\lesssim b$ and $b\lesssim a$. For a linear operator $T$, the notation $\|T\|_p$ denotes its operator norm on $L^p$.

\section{Preliminaries}
In this section, we shall establish some basic lemmas. A basic property of polynomials will be needed in our application of the operator van der Corput lemma below. The first part of the following lemma is previously known; see Phong-Stein \cite{PS1994,PS1998}.

\begin{lemma}\label{sec2 poly type func lemma}
Assume $P$ is a polynomial in $\bR$ with degree not greater than $d$. Then there exists a constant $C=C(d)$, depending only on $d$, such that for any bounded interval $I\subseteq \bR$,
\begin{equation}\label{sec2 esti of poly type func}
\sum_{k=0}^{d}|I^{\ast}|^k\sup_{x\in I^{\ast} }
|P^{(k)}(x)|
\leq
C \sup_{x\in I } |P(x)|,
\end{equation}
where $I^{\ast}$ is the interval with the same center as $I$ but with twice the length of $I$.

Moreover, assume further $P$ is real-valued. If, in addition, there exist two numbers $\mu>0,A>0$ and a bounded interval $J$ such that $\mu\leq |P(x)| \leq A\mu$ for all $x\in J$, then for all $z\in\mathbb{C}$ and all intervals $I\subseteq J$, $Q_z(x)=|P(x)|^z$ also satisfies the above estimate with a constant $C=C(d,A,z)$ and $I^{\ast}$ replaced by $I^{\ast}\cap J$.
\end{lemma}
\begin{proof}
By translation and scaling, we may assume $I=[0,1]$. Let $V$ be the space of all polynomials in $\bR$ with degree not greater than $d$. Then $V$ is a finite-dimensional vector space. It is clear that both sides of (\ref{sec2 esti of poly type func}) are norms on $V$. Hence the desired estimate follows.

By our assumption, $P$ has fixed sign on the interval $J$. For example, assume $P>0$. By induction, we can show that all derivatives of $Q_z$ have the following form:
\begin{equation*}
\frac{d^N}{dx^N}Q_z(x)
=
|P(x)|^{z-N}
\sum_{} C_{k_1,k_2,\cdots,k_N}(z)
P^{(k_1)}(x)P^{(k_2)}(x) \cdots P^{(k_N)}(x)
\end{equation*}
where the summation is taken over all integers $k_i\geq 0$ satisfying $k_1+k_2+\cdots+k_N=N$. For any interval $I\subseteq J$, we can apply the inequality (\ref{sec2 esti of poly type func}) to obtain the desired estimate.
\end{proof}

Now we give the operator van der Corput lemma due to Phong-Stein \cite{PS1994,PS1997} and Phong-Stein-Sturm \cite{PSS2001}.

The crucial notion of curved trapezoid is given as follows.
\begin{defn}
If $g$ and $h$ are two monotone functions on an interval $[a,b]$, then
$$
\Omega=\{(x,y)\mid a\leq x\leq b,~g(x)\leq y \leq h(x)\}
$$
is said to be a curved trapezoid.
\end{defn}

\begin{lemma}\label{operator van der Corput}
Let $T_{\lambda}$ be a {\rm(1+1)}-dimensional oscillatory integral operator as in {\rm(\ref{sec1 general oio})}, where $S$ is a real-valued polynomial in $\bR^2$ and $\varphi$ is supported in a curved trapezoid
$\Omega$. If the Hessian of the phase satisfies $\mu\leq |S''_{xy}(x,y)|\leq A\mu$ on $\Omega$ for two positive numbers $\mu,A>0$,
then there exists a constant $C=C(deg(S),A)$ such that
$$\|T_{\lambda}f\|_{L^2}
\leq C \left( \sum\limits_{k=0}^{2}
\mathop{\sup}\limits_{\Omega}
\big( \delta_{\Omega,h}(x) \big)^k |\partial_y^k\varphi(x,y)| \right)
|\lambda\mu|^{-1/2}\|f\|_{L^2},$$
where $\delta_{\Omega,h}(x)$ denotes the length of the vertical cross-section $I_{\Omega,h}(x)=\{y\mid (x,y)\in\Omega\}$.
\end{lemma}

The following simple version of almost orthogonality principle will be frequently used in this paper; see Phong-Stein-Sturm \cite{PSS2001} for its proof. For a general set $A$, the notation $\chi_{A}$ denotes the characteristic function of $A$.
\begin{lemma}\label{sec2 almost ortho lemma}
Let $K$ be a Lebesgue measurable function in $\bR^{n_X}\times \bR^{n_Y}$ and $m_{d}$ the Lebesgue measure on $\bR^d$. Assume that there are measurable sets $A_i\subseteq \bR^{n_X}$ and $B_i\subseteq \bR^{n_Y}$ such that $m_{n_X}(A_i\cap A_j)= 0$ and $m_{n_Y}(B_i\cap B_j)= 0$ for $|i-j|\geq N_0$. Let $K_i(x,y)=K(x,y)\chi_{A_i}(x)\chi_{B_i}(y)$ and $T$ (respectively, $T_i$) be the integral operator associated with the kernel $K$ (respectively, $K_i$). If $T=\sum_{i} T_i$, then for all $1\leq p \leq \infty$ we have
\begin{equation*}
\|T\|_p \leq N_0 \sup_{i} \|T_i\|_p,
\end{equation*}
where $\|T\|_p$ and $\|T_i\|_p$ denote the operator norms of $T$ and $T_i$ respectively, as operators from $L^p(\bR^{n_Y})$ into $L^p(\bR^{n_X})$.
\end{lemma}

As a useful interpolation technique, we also need the following interpolation with change of power weights. An earlier version of this lemma appeared in Pan, Sampson and Szeptycki \cite{pansampson}; see also \cite{ShiYan, Shi}.
\begin{lemma}\label{sec2 interp change of meas}
Let $T$ be a sublinear operator mapping simple functions in $\bR^{n_Y}$, defined with respect to Lebesgue measure, into measurable functions in $\bR^{n_X}$. Assume that there are constants $A,B>0$ and $a\neq -\frac{1}{2n_X}$ such that, for all simple functions in
$\bR^{n_Y}$,
\begin{enumerate}
\item[{\rm (i)}]  $\|Tf\|_{L^{\infty}(dx)}\leq A \|f\|_{L^1(dy)}${\rm ;}
\item[{\rm (ii)}] $\||x|^aTf\|_{L^2 (dx)}\leq B \|f\|_{L^2(dy)}$.
\end{enumerate}
Then for any $\theta\in (0,1)$, there exists a constant $C=C(a,n_X,\theta)$ such that
\[
\||x|^{a\theta-(1-\theta)n_X}Tf\|_{L^p(dx)}\leq CA^{1-\theta}B^{\theta}\|f\|_{L^p(dy)},
~~~\frac1{p}=\frac{\theta}{2}+1-\theta.
\]
\end{lemma}

\begin{proof}
Define a measure $d\mu=|x|^{c}dx$ on $\bR^{n_X}$ and an operator $Wf(x)=|x|^bTf(x)$, where $b,c\in\bR$ are to be determined. Here we need a simple fact that $|x|^{b}$ belongs to $L^{1,\infty}(|x|^{c}dx)$ in $\bR^{n_X}$ if and only if $b+c=-n_X$ and $b\neq 0$. Now we first consider $b>0$. For any $\lambda>0$, $|x|^b>\lambda$ is equivalent to $|x|>\lambda^{1/b}$. Hence
\begin{eqnarray*}
d\mu(\{x\in\bR^{n_X}: |x|^b>\lambda\})
=
\int_{|x|>\lambda^{1/b}} |x|^{-n_X-b}dx
=
C(n_X,b)\lambda^{-1}.
\end{eqnarray*}
Thus our claim is true for $b>0$. Similarly, we can show the claim for $b<0$.

With the above result, we take $b=2a+n_X$ and $c=-2a-2n_X$. By Assumptions (i) and (ii), $W$ is bounded from $L^1(\bR^{n_Y};dy)$ and $L^2(\bR^{n_Y};dy)$ into $L^{1,\infty}(\bR^{n_X};d\mu)$ and $L^2(\bR^{n_X};d\mu)$, respectively. By the Marcinkiewicz interpolation theorem, we see that for any $0<\theta<1$, there exists a constant $C=C(a,n_X,\theta)$ such that
\begin{equation*}
\|Wf\|_{L^p(|x|^cdx)}
\leq
C A^{1-\theta} B^{\theta} \|f\|_{L^p(dy)}
\end{equation*}
with $\frac{1}{p}=\theta/2+1-\theta=1-\theta/2$. In other words,
$$
\||x|^{a\theta-(1-\theta)n_X}Tf\|_{L^p(dx)}
\leq
C A^{1-\theta} B^{\theta} \|f\|_{L^p(dy)}
$$
for $\frac{1}{p}=1-\theta/2$. The proof of the lemma is complete.
\end{proof}

\section{Sharp $L^p$ estimates}
In this section, we shall establish $L^p$ estimates for oscillatory integral operators with homogeneous polynomial phases satisfying the rank one condition. The corresponding endpoint $L^p$ estimates will be given in Section 4. Now we first introduce the concept of rank one condition due to Greenleaf, Pramanik and Tang \cite{GPT}.

\begin{defn}
Let $S$ be a homogeneous polynomial in $\bR^{n_X}\times\bR^{n_Y}$ with real coefficients. We say that $S$ satisfies the rank one condition if ${\bf{rank}}( {\rm Hess} (S)(x,y))\geq 1$ away from the origin, i.e., the system of equations $\partial_{x_i}\partial_{y_j}S(x,y)=0$ does not have a solution $(x,y)\in \bR^{n_X}\times\bR^{n_Y}\backslash \{(0,0)\}$.
\end{defn}

\noindent Under the rank one condition, we can state our main result in this section as follows.
\begin{theorem}\label{sec3 main thm 1}
Assume $S$ is a homogeneous polynomial in $\bR^{n_X}\times\bR^{n_Y}$ with real coefficients and degree $d>n_X+n_Y$. Let $T_{\lambda}$ be the oscillatory integral operator as in {\rm(\ref{sec1 general oio})}. If $S$ satisfies the rank one condition, then for $p$ in the following range
\begin{equation}\label{sec3 range lp boundedness}
\frac{d-n_Y+n_X}{d-n_Y}
<p<
\frac{d-n_X+n_Y}{n_Y},
\end{equation}
there exists a constant $C=C(S,\varphi,p)$ such that
\begin{equation}\label{sec3 main Lp esti}
\|T_{\lambda}f\|_{L^p}\leq C|\lambda|^{-\gamma}\|f\|_{L^p},~~~
\gamma=\frac{n_X}{d}\cdot\frac{1}{p}
+\frac{n_Y}{d}\cdot\frac{1}{p'},
\end{equation}
where $p'$ is the conjugate exponent of $p$, i.e., $1/p'=1-1/p$. Moreover, this estimate is sharp provided that the cut-off $\varphi$ does not vanish near the origin.
\end{theorem}

\begin{proof}
Our proof will be divided into two steps.\\

\textbf{Step 1}. Sharpness of the decay rate.\\

Assume $\varphi(0,0)\neq 0$ and $|\lambda|$ is sufficiently large. Let $f(y)=\chi_{\{|y|\leq \epsilon_0 |\lambda|^{-1/d}\}}$ for some small $\epsilon_0>0$. For $x\in\bR^{n_X}$ near the origin, $|x|\leq \epsilon_0 |\lambda|^{-1/d}$, we obtain $|T_{\lambda}f(x)|\gtrsim |\lambda|^{-n_Y/d}$. Hence
\begin{eqnarray*}
\|T\|_p
\geq
\|T_{\lambda}f\|_{L^p}/\|f\|_{L^p}
\gtrsim  |\lambda|^{-n_Y/d}
|\lambda|^{-\frac{n_X}{dp}}/
|\lambda|^{-\frac{n_Y}{dp}}=|\lambda|^{-\gamma}.
\end{eqnarray*}

\textbf{Step 2.} Proof of the optimal decay estimate.

For $r>0$, we use $B_r((x,y))$ to denote the ball in $\bR^{n_X+n_Y}$ with radius $r$ and center $(x,y)$. Let $S^{n_X+n_Y-1}$ be the unit sphere centered at the origin in $\bR^{n_X+n_Y-1}$. In the following argument, we need a partition of unity on $S^{n_X+n_Y-1}$. For our purpose, we shall first give an appropriate open cover $\{\mathcal{O}_{\alpha}\}$ of $S^{n_X+n_Y-1}$.

We first choose a sufficiently small $r>0$ such that for each point $x^{\ast}:=(x,0)\in S^{n_X+n_Y-1}$, there exists a pair of indices $(i,j)$, $1\leq i \leq n_X$, $1\leq j \leq n_Y$, such that the mixed derivative $\partial_{x_i}\partial_{x_j}S(x,y)$ does not change sign and its absolute value is comparable to a positive constant for all $(x,y)\in B_r(x^{\ast})\cap S^{n_X+n_Y-1}$. Since $S^{n_X-1}$ is compact, we can select finitely many points $x_1,x_2,\cdots,x_M\in S^{n_X-1}$ such that $\cup_iB_r(x_i^{\ast})\supseteq S^{n_X-1}\times \{0_{n_Y}\}$. Here the notation $0_m$ denotes the origin in $\bR^m$. In this way, if $r>0$ is small enough, we can choose $y_1,\cdots,y_N\in S^{n_Y-1}$ such that the union $\cup_i B_r(y_i)$ covers $S^{n_Y-1}$ and some mixed derivative $\partial_{x_s}\partial_{y_t}S$ does not vanish on each given ball $B_r(y_i^{\ast})$.

Similarly, there exists a small number $\rho>0$ and finitely many points $w_1,w_2,\cdots,w_K\in S^{n_X+n_Y-1}$ such that the following three properties hold:

\textbf{(i)} The union of $B_{\rho}(w_k)\cap S^{n_X+n_Y-1}$ covers the complement of $\cup_{i,j}\big(B_r(x_i^{\ast})\cup B_r(y_j^{\ast})\big)\cap S^{n_X+n_Y-1}$ relative to $S^{n_X+n_Y-1}$, i.e.,
\begin{equation*}
\bigcup_{k=1}^K \Big(B_{\rho}(w_k)\cap S^{n_X+n_Y-1}\Big)
\supseteq
\Big(
\cup_{i,j}\big(B_r(x_i^{\ast})\cup B_r(y_j^{\ast})\big)
\Big)^c\cap S^{n_X+n_Y-1};
\end{equation*}

\textbf{(ii)} Each $B_{\rho}(w_k)$ does not intersect both $B_{r/2}(x_i^{\ast})\cap S^{n_X+n_Y-1}$ and
$B_{r/2}(y_j^{\ast})\cap S^{n_X+n_Y-1}$ for all $i$ and $j$, i.e.,
\begin{equation*}
B_{\rho}(w_k)\cap B_{r/2}(x_i^{\ast})\cap S^{n_X+n_Y-1}=\emptyset,~~~
B_{\rho}(w_k)\cap B_{r/2}(y_j^{\ast})\cap S^{n_X+n_Y-1}=\emptyset;
\end{equation*}

\textbf{(iii)} For each $1\leq k \leq K$, there exists a pair of indices $(s,t)$ such that $\partial_{x_s}\partial_{y_t}S$ has fixed sign on $B_{\rho}(w_k)$ and its absolute value is bounded from both above and below by positive constants.\\

Let $\{\mathcal{O}_{\alpha}\}$ be the open cover consisting of $B_{r}(x_i^{\ast})\cap S^{n_X+n_Y-1}$, $B_r(y_j^{\ast})\cap S^{n_X+n_Y-1}$ and $B_{\rho}(w_k)\cap S^{n_X+n_Y-1}$. As discussed above, the union of $\mathcal{O}_{\alpha}$ covers the sphere $ S^{n_X+n_Y-1}$. Corresponding to this open cover, we can now construct a partition of unity $\{\Psi_{\alpha}\}$ such that each $\Psi_{\alpha}$ is homogeneous of degree zero, $\Psi_{\alpha}|_{S^{n_X+n_Y-1}}\in C_0^{\infty}(\mathcal{O}_{\alpha})$ and $\sum_{\alpha}\Psi_{\alpha}(x)=1$ for all $x\in S^{n_X+n_Y-1}$.

For each $k\in \mathbb{Z}$ and $\alpha$, we define $T_{\lambda,\alpha,k}$ as $T_{\lambda}$ in (\ref{sec1 general oio}) by insertion of $\Psi_{\alpha}(x,y)\Phi(x/2^k,y/2^k)$ into the cut-off of $T_{\lambda}$, i.e.,
\begin{equation}\label{sec3 def of Tkalpha}
T_{\lambda,\alpha,k}f(x)=\int_{\bR^{n_Y}} e^{i\lambda S(x,y)}
\Psi_{\alpha}(x,y)
\Phi\left(\frac{x}{2^k},\frac{y}{2^k}\right)
\varphi(x,y)f(d)dy,
\end{equation}
where $\Phi\in C_0^{\infty}$ is supported in the annulus $1/2 \leq |(x,y)|\leq 2$ such that $\sum_{k}\Phi(x/2^k,y/2^k)=1$ for all $(x,y)$ away from the origin.

In what follows, we shall establish the sharp $L^p$ estimate in the theorem. It is more convenient to divide our argument into three cases.\\

\textbf{Case 1} $\supp(\Psi_{\alpha})\cap S^{n_X+n_Y-1}\subseteq B_{\rho}(w_i)$ for some $i$.\\

In this case, $\Psi_{\alpha}|_{S^{n_X+n_Y-1}}$ is supported in an open subset of $S^{n_X+n_Y-1}$ (with subset topology) which does not intersect both the $X-$space and the $Y-$space. Let $T_{\lambda,\alpha}=\sum_kT_{\lambda,\alpha,k}$. Then $T_{\lambda,\alpha}$ is supported in a cone, with vertex at the origin, which does not intersect both the
$X-$space and the $Y-$space away from the origin. Hence for $(x,y)$ in the support of $T_{\lambda,\alpha,k}$, we have $|x|\approx |y|\approx 2^k$. It follows immediately that if $|k-l|\geq A_0$ for some large positive integer $A_0$, then
\begin{equation*}
\mathbb{P}_X(\supp(T_{\lambda,\alpha,k})
\cap
\supp(T_{\lambda,\alpha,l}))=\emptyset,~~~
\mathbb{P}_Y(\supp(T_{\lambda,\alpha,k})
\cap
\supp(T_{\lambda,\alpha,l}))=\emptyset,
\end{equation*}
where $\mathbb{P}_X$ and $\mathbb{P}_Y$ are the projections from $\bR^{n_X+n_Y}$ onto the $X-$space $\bR^{n_X}$ and the $Y-$space $\bR^{n_Y}$, respectively.

By the almost orthogonality principle in Lemma \ref{sec2 almost ortho lemma}, we have $\|T_{\lambda,\alpha}\|_p\leq A_0\sup_k\|T_{\lambda,\alpha,k}\|_p$ for all $1\leq p \leq \infty$. To establish our desired estimate, it suffices to prove that each $T_{\lambda,\alpha,k}$ satisfies the $L^p$ estimate (\ref{sec3 main Lp esti}).

By our assumption, there exist indices $i_{\alpha}$ and $j_{\alpha}$ such that $\partial_{x_{i_{\alpha}}} \partial_{y_{j_{\alpha}}}S(x,y)$ does not change its sign on the support of $T_{\lambda,\alpha,k}$ and its absolute value is bounded from above and below by positive constants. Hence we can apply Lemma \ref{operator van der Corput} with respect to the variables $x_{i_{\alpha}}$ and $y_{j_{\alpha}}$, letting other variables fixed temporarily, and make use of the Schur test, with respect to other variables, to obtain
\begin{eqnarray}\label{sec3 main L2 esti 1}
\|T_{\lambda,\alpha,k}\|_2
&\leq&
C 2^{k(n_X+n_Y-2)/2}\Big(|\lambda|2^{k(d-2)}\Big)^{-1/2}
\nonumber\\
&=&
C
2^{k(n_X+n_Y)/2}\Big(|\lambda|2^{kd}\Big)^{-1/2}.
\end{eqnarray}
On the other hand, since $|x|\approx |y|\approx 2^k$ on the support of $T_{\lambda,\alpha,k}$, it is clear that
\begin{equation*}
\|T_{\lambda,\alpha,k}\|_1
\leq
C
2^{kn_X}.
\end{equation*}
Let $\theta=\frac{2n_X}{d+n_X-n_Y}\in (0,1)$. With this $\theta$, we use the Riesz-Th\"orin interpolation theorem to obtain
\begin{equation}\label{sec3 main esti 1a}
\|T_{\lambda,\alpha,k}\|_p
\leq
C
|\lambda|^{-\frac{n_X}{d+n_X-n_Y}}
\end{equation}
where
\begin{equation*}
\frac1p=\frac{\theta}{2}+1-\theta
=1-\frac{n_X}{d+n_X-n_Y}=\frac{d-n_Y}{d+n_X-n_Y}.
\end{equation*}
By a duality argument, we also have
\begin{equation}\label{sec3 main esti 1b}
\|T_{\lambda,\alpha,k}\|_q
\leq
C
|\lambda|^{-\frac{n_Y}{d-n_X+n_Y}},
~~~q=\frac{d-n_X+n_Y}{n_Y}.
\end{equation}
In fact, $T_{\lambda,\alpha,k}$ satisfies the $L^{\infty}$ estimate $\|T_{\lambda,\alpha,k}\|_{\infty}
\leq
C
2^{kn_Y}$. Interpolation this with the $L^2$ inequality (\ref{sec3 main L2 esti 1}) gives the above $L^q$ estimate.

By interpolation, we see that each $T_{\lambda,\alpha,k}$ satisfies the estimate (\ref{sec3 main Lp esti}) uniformly. By the almost orthogonality described as above, we see that $T_{\lambda,\alpha}$ also satisfies the desired estimate.\\

\textbf{Case 2} $\supp(\Psi_{\alpha})\cap S^{n_X+n_Y-1}\subseteq B_{\rho}(x_i^{\ast})$ for some $i$.\\

In this case, we have $|x|\gtrsim |y|$ and $|x|\approx 2^k$ in the support of $T_{\lambda,\alpha,k}$. The almost orthogonality in Case 1 is not true now. By insertion of the damping factor $|x|^z$, we shall consider the following damped operator $W_{\lambda,\alpha,k}^z$ associated with $T_{\lambda,\alpha,k}$,
\begin{equation}\label{sec3 def of Wkalpha z}
W_{\lambda,\alpha,k}^zf(x)=\int_{\bR^{n_Y}} e^{i\lambda S(x,y)}
\Psi_{\alpha}(x,y)
\Phi\left(\frac{x}{2^k},\frac{y}{2^k}\right)
|x|^z\varphi(x,y)f(d)dy.
\end{equation}
Let $W_{\lambda,\alpha}^z=\sum_kW_{\lambda,\alpha,k}^z$. With this definition, it is easy to see that $W_{\lambda,\alpha}^z$ is bounded from $L^1(\bR^{n_Y})$ into $L^{1,\infty}(\bR^{n_X})$ provided that $z$ has real part $\Re(z)=-n_X$.

In what follows, our main purpose is to establish $L^2$ damping estimates for $W_{\lambda,\alpha}^z$. Let $\sigma_0=\frac{d-n_X-n_Y}{2}$. One will see that $\sigma_0$ is the critical exponent in the following damping estimates:
\begin{equation}\label{sec3 damping esti 2}
\|W_{\lambda,\alpha}^z\|_2\lesssim
\begin{cases}
|\lambda|^{-1/2},~~~~~~ ~~~~~~ ~~~~\quad~~
\sigma:=\Re(z)>\sigma_0;  \\
|\lambda|^{-1/2}\log(2+|\lambda|),
~~~~~~\sigma:=\Re(z)=\sigma_0;\\
|\lambda|^{-\frac{\sigma}{d}-\frac{n_X+n_Y}{2d}},
~~~~~-\min\{n_X,n_Y\}<\sigma:=\Re(z)<\sigma_0.
\end{cases}
\end{equation}
In these damping estimates, the implicit constants can take the form $C(1+|z|^2)$ with $C$ independent of $\lambda$ and $z$.

Now we turn to prove (\ref{sec3 damping esti 2}). As in Case 1, we use the operator version van der Corput lemma in Lemma \ref{operator van der Corput} to obtain
\begin{equation*}
\|W_{\lambda,\alpha,k}^z\|_2
\leq
C(1+|z|^2)\Big(|\lambda|2^{kd}\Big)^{-1/2}
2^{k\sigma} 2^{k(n_X+n_Y)/2}.
\end{equation*}
On the other hand, the size of the support of $W_{\lambda,\alpha,k}^z$ implies
\begin{equation*}
\|W_{\lambda,\alpha,k}^z\|_2
\leq
C 2^{k\sigma} 2^{k(n_X+n_Y)/2}.
\end{equation*}
For $z$ with real part $\sigma$, we have
\begin{eqnarray*}
\|W_{\lambda,\alpha}^z\|_2
&\lesssim&
\sum_{k}\min\Big\{\Big(|\lambda|2^{kd}\Big)^{-1/2}
2^{k\sigma} 2^{k(n_X+n_Y)/2},2^{k\sigma} 2^{k(n_X+n_Y)/2}\Big\}\\
&\lesssim&
\sum_{2^{kd}|\lambda|>1}\Big(|\lambda|2^{kd}\Big)^{-1/2}
2^{k\sigma} 2^{k(n_X+n_Y)/2}
+
\sum_{2^{kd}|\lambda|\leq 1}2^{k\sigma} 2^{k(n_X+n_Y)/2}.
\end{eqnarray*}
If $\sigma>-\min\{n_X,n_Y\}$, then it follows from $\sigma+(n_X+n_Y)/2>0$ that the above second summation is bounded by a constant multiple of $|\lambda|^{-\sigma/d-(n_X+n_Y)/(2d)}$.
For the first summation, we obtain an upper bound $\lesssim |\lambda|^{-1/2}$ for $\sigma>\sigma_0$. In the strip $-\min\{n_X,n_Y\}<\sigma\leq \sigma_0$, the first summation satisfies the same estimate as the second one, up to a logarithmic term for $\sigma=\sigma_0$. Combining these estimates, we obtain (\ref{sec3 damping esti 2}).

For $p$ in the range (\ref{sec3 range lp boundedness}) and $p\leq 2$, the parameter $\theta$ for which $1/p=\theta/2+1-\theta$ must satisfy
$\theta_0=\frac{2n_X}{d+n_X-n_Y}<\theta\leq 1$. It should be pointed out that $p$ is just the left endpoint in the interval (\ref{sec3 range lp boundedness}) if $\theta=\theta_0$. If $-n_X(1-\theta)
+\sigma\theta=0$ then $0<\sigma<\sigma_0$.

Recall that $W_{\lambda,\alpha}^z$ is bounded from $L^1(\bR^{n_Y})$ into $L^{1,\infty}(\bR^{n_X})$ for $\Re(z)=-n_X$. By interpolation in Lemma \ref{sec2 interp change of meas}, we have
\begin{equation*}
\|W_{\lambda,\alpha}^z\|_p
\lesssim
|\lambda|^{-\l(\frac{\sigma}{d}
+\frac{n_X+n_Y}{2d}\r)\theta},~~~\Re(z)=0.
\end{equation*}
Here $\theta$ satisfies $1/p=1-\theta/2$ and $-n_X(1-\theta)+\sigma\theta=0$. Thus $\theta=2/p'$ and $\sigma=n_X(p'/2-1)$ with $p'$ being the conjugate exponent of $p$. It follows that the decay exponent above is equal to
\begin{eqnarray*}
\l(\frac{\sigma}{d}
+\frac{n_X+n_Y}{2d}\r)\theta
&=&
\l[\frac{n_X}{d}  \l( \frac{p'}{2}-1 \r)  +  \frac{n_X+n_Y}{2d}\r]\cdot \frac{2}{p'}\\
&=&
\frac{n_X}{dp}+\frac{n_Y}{dp'},
\end{eqnarray*}
as desired.
\\

\textbf{Case 3} $\supp(\Psi_{\alpha})\cap S^{n_X+n_Y-1}\subseteq B_{\rho}(y_j^{\ast})$ for some $j$.\\

In the support of $T_{\lambda,\alpha,k}$, we have $|y| \approx 2^k$ and $|y|\gtrsim |x|$. The argument in this case is in many ways like that of Case 2. Define the damped oscillatory integral operator $W^z_{\lambda,\alpha,k}$ as in Case 2 with the damping factor $|x|^z$ replaced by $|y|^z$. As above, $W_{\lambda,\alpha}^z=\sum_k W_{\lambda,\alpha,k}^z$ satisfies the same damping $L^2$ estimates. The only difference lies in the situation $\Re(z)=-n_X$. In fact, by Fubini's theorem, it is easy to see that $W_{\lambda,\alpha}^z$, with $\Re(z)=-n_X$, is bounded from $L^1(\bR^{n_Y})$ into $L^1(\bR^{n_X})$. More precisely, we have
\begin{eqnarray*}
\int_{\bR^{n_X}} |W_{\lambda,\alpha}^z f(x)|dx
&\lesssim&
\int_{\bR^{n_X}} \left( \int_{|y|\gtrsim |x|} |y|^{-n_X} |f(y)| dy\right) dx\\
&=&
\int_{\bR^{n_Y}} |y|^{-n_X} |f(y)| \left( \int_{|x|\lesssim |y |} dx\right) dy\\
&\lesssim&
\int_{\bR^{n_Y}} |f(y)| dy.
\end{eqnarray*}
By interpolation as in Case 2, $T_{\lambda,\alpha}$ satisfies the desired $L^p$ estimate.

Combining all above results, we complete the proof of the theorem.
\end{proof}

We can define a class of more general damped oscillatory integral operators associated with $T_{\lambda}$ in (\ref{sec1 general oio}). Let $W_{\lambda,D}^z$ be given by
\begin{equation}\label{sec3 general damped oio}
W_{\lambda,D}^zf(x)
=
\int_{\bR^{n_Y}} e^{i\lambda S(x,y)}
|D(x,y)|^z \varphi(x,y)f(d)dy
\end{equation}
where $z\in\mathbb{C}$ and $D$ is a damping function. Under the rank one condition, we have the following
\begin{theorem}\label{sec3 general dampin thm}
Assume $S$ is a real-valued homogeneous polynomial in $\bR^{n_X}\times \bR^{n_Y}$ with degree $d>n_X+n_Y$. Let $D$ be a real-valued homogeneous polynomial which does not vanish away from the origin. If $S$ satisfies the rank one condition, then there exists a constant $C=C(S,D,\varphi)$ such that
\begin{equation}\label{sec3 general damping esti}
\|W_{\lambda,D}^z\|_2
\leq
C(1+|z|)^2
\begin{cases}
|\lambda|^{-1/2},~~~~~~ ~~~~~~~~~\quad
\Re(z)>\frac{d-n_X-n_Y}{2d_D};  \\
|\lambda|^{-1/2}\log(2+|\lambda|),
~~~~\Re(z)=\frac{d-n_X-n_Y}{2d_D};\\
|\lambda|^{-\frac{d_D}{d}\Re(z)-\frac{n_X+n_Y}{2d}},
~~-\frac{\min\{n_X,n_Y\}}{d_D}<\Re(z)<\frac{d-n_X-n_Y}{2d_D},
\end{cases}
\end{equation}
where $d_D$ is the degree of $D$.
\end{theorem}
\begin{remark}
Under the assumptions in the theorem, we can take $D(x,y)=|x|^2+|y|^2$. Also, the damping function can be chosen as the Hilbert-Schmidt norm of the Hessian of the phase function, i.e.,
$$D(x,y)
=
\Big(\sum_{i=1}^{n_X}\sum_{j=1}^{n_Y}
|\partial_{x_i}\partial_{y_j}S(x,y)|^2 \Big)^{1/2}.$$
Generally, $D$ is not a polynomial but the above damping estimates are still true with $d_D=d-2$. In the special case $n_X=n_Y$, the damping estimates in the theorem, with $D$ being the Hilbert-Schmidt norm of the Hessian of $S$, were proved by Xu-Yan {\rm\cite{xuyan}}. For $(1+1)-$dimensional damping estimates with $D=S''_{xy}$, we refer the reader to Seeger {\rm\cite{seeger2}} and Phong-Stein {\rm\cite{PS1998}}.
\end{remark}

The proof of Theorem \ref{sec3 general dampin thm} is the same as that of the damping estimates (\ref{sec3 damping esti 2}). We omit the details here.

\section{Endpoint $L^p$ estimates}

Until now, we do not know whether endpoint $L^p$ estimates in Theorem \ref{sec3 main thm 1} are true or not. Our proof in Section 3 breaks down since it will produce a logarithmic term. More precisely, we only have
\begin{equation}
\|T_{\lambda}f\|_{L^p}\leq C|\lambda|^{-\gamma} \log^{\delta}(2+|\lambda|)
\|f\|_{L^p},~~~
p\in\left\{ \frac{d-n_Y+n_X}{d-n_Y}, \frac{d-n_X+n_Y}{n_Y}  \right\},
\end{equation}
where $\delta>0$ is a number in $(0,1)$ and $\gamma$ is given by (\ref{sec3 main Lp esti}). In this section, our purpose is to remove this logarithmic term under certain assumptions.

We first introduce a useful notion of nondegeneracy for the phase $S$.

\begin{defn}\label{sec4 radially nondegeneracy}
Assume $G$ is a continuously differentiable function from $\bR^{m}$ into $\bR^n$. Then $G$ is said to be radially nondegenerate if
$(x\cdot \nabla_x)G(x)\neq 0$ for all $x\neq 0$.
\end{defn}

\begin{lemma}
Assume $S$ is a homogeneous polynomial in $\bR^{n_X}\times\bR^{n_Y}$. Let $P_i(x)=\partial_{y_i}S(x,y)|_{y=0}$ and $Q_j(y)=\partial_{x_j}S(x,y)|_{x=0}$. Then $\nabla_{y}S(x,0)$ is radially nondegenerate in the $X-$space $\bR^{n_X}$ if and only if $(P_1(x),P_2(x),\cdots,P_{n_Y}(x))\neq 0$ for $x\neq 0$. Similarly, $\nabla_{x}S(0,y)$ is radially nondegenerate in the $Y-$space $\bR^{n_Y}$ if and only if $(Q_1(y),Q_2(y),\cdots,Q_{n_X}(y))\neq 0$ for $y\neq 0$.
\end{lemma}

\begin{proof}
Denote by $d$ the degree of $S$. Since the notion of radial nondegeneracy involves partial derivatives of second order, our assumptions imply that $d\geq 2$ and $P_i, Q_j$ are homogeneous polynomials of degree $d-1$. By Euler's formula for homogeneous functions,
\begin{equation*}
(x\cdot\nabla_x)\partial_{y_i}S(x,0)
=(x\cdot\nabla_x)P_i(x)=(d-1)P_i(x).
\end{equation*}
Similarly, $(y\cdot\nabla_y)\partial_{x_j}S(0,y)=(d-1)  Q_j(y)$. By Definition \ref{sec4 radially nondegeneracy}, the statement in the lemma follows immediately.
\end{proof}

\begin{remark}
In the $X$ and $Y$ spaces, the rank one condition is slightly weaker than the radial nondegeneracy of $\nabla_yS(x,0)$ and $\nabla_xS(0,y)$. For example, consider the rank one condition in the $X-$space. Since $y=0$ in the $X-$space, the rank one condition, at the point $(x,0)$ with $x\neq 0$, implies $\nabla_xP_i(x)\neq 0$ for some $i$. However, the radial nondegeneracy of $\nabla_yS(x,0)$ is equivalent to $P_i(x)=(d-1)^{-1}(x\cdot\nabla_x)P_i(x)\neq 0$ for all $i$.
\end{remark}

With the concept of radial nondegeneracy, we are able to establish the endpoint $L^p$ estimates in Theorem \ref{sec3 main thm 1}.

\begin{theorem}
Assume $S$ is a real-valued homogeneous polynomial with degree $d>n_X+n_Y$. Suppose $S$ satisfies the following two conditions:

\noindent {\rm (i)} $S$ satisfies the rank one condition in $\bR^{n_X} \times \bR^{n_Y}$.

\noindent {\rm (ii)} $\nabla_yS(x,0)$ and $\nabla_xS(0,y)$ are radially nondegenerate in $\bR^{n_X} $ and $\bR^{n_Y}$, respectively.

\noindent Then $T_{\lambda}$ in {\rm (\ref{sec1 general oio})} satisfies the $L^p$ estimate {\rm (\ref{sec3 main Lp esti})} for
\begin{equation*}
\frac{d-n_Y+n_X}{d-n_Y}
\leq p \leq
\frac{d-n_X+n_Y}{n_Y}.
\end{equation*}
Moreover, under the assumptions in Theorem {\rm \ref{sec3 general dampin thm}}, the damping estimates {\rm (\ref{sec3 general damping esti})} are still true without the logarithmic term $\log(2+|\lambda|)$.
\end{theorem}
\begin{remark}
There may be no phases satisfying Assumptions (i) and (ii) if $S$ has an even degree. For example, let $n_X\geq 2$ and $n_Y=1$. Assume $S$ is a homogeneous polynomial and its degree $d$ is even. Then $\partial_yS(x,0)$ is homogeneous in $\bR^{n_X}$ and its degree is odd. Since $n_X\geq 2$, one can see that $\partial_y S(x,0)$ has zeros away from the origin.

However, if $n_X=n_Y$ or if $d\geq 2$ is odd, then homogeneous phases $S$ satisfying (i) and (ii) always exist. We take two examples, due to Greenleaf, Pramanik and Tang \cite{GPT}, in the following:

{\rm (1)} If $n_X=n_Y$,
$S(x,y)=\frac{1}{d-1}\left(\sum_{i=1}^{n_X}x_i^{d-1}y_i
+\sum_{i=2}^{n_X}x_{i-1}y_i^{d-1}+x_{n_X}y_1^{d-1}\right).$

{\rm (2)} Assume $n_X>n_Y$ and $d\geq 2$ is odd. For example, we can take $S$ as
$$S(x,y)=\frac{1}{d-1}\left(\sum_{i=1}^{n_Y}x_i^{d-1}y_i
+\sum_{i=2}^{n_Y}x_{i-1}y_i^{d-1}+x_{n_Y}y_1^{d-1}
+\sum_{i=n_Y+1}^{n_X}\sum_{j=1}^{n_Y}x_i^{d-1}y_j\right).$$
\end{remark}

\begin{proof}
For clarity, we shall divide our proof into two steps.
The first step is to prove the $L^2$ damping estimates, from which the desired endpoint $L^p$ estimates follow immediately. \\

\textbf{Step 1}. Proof of $L^2$ damping estimates.\\

As in our proof of Theorem \ref{sec3 main thm 1}, we shall consider three cases separately. Let $W_{\lambda,D,\alpha,k}^z$ be defined as $T_{\lambda,\alpha,k}$ by insertion of a damping factor $|D(x,y)|^z$, i.e.,
\begin{equation*}
W_{\lambda,D,\alpha,k}^zf(x)=\int_{\bR^{n_Y}} e^{i\lambda S(x,y)}
\Psi_{\alpha}(x,y)
\Phi\left(\frac{x}{2^k},\frac{y}{2^k}\right)
|D(x,y)|^z
\varphi(x,y)f(d)dy.
\end{equation*}
Taking summation over $k$, we define $W_{\lambda,D,\alpha}^z
=\sum_kW_{\lambda,D,\alpha,k}^z$.
Since $W_{\lambda,D,\alpha,k}^z=0$ for all $k\gtrsim 1$, we assume $k\lesssim 1$ from now on. \\

\textbf{Case 1} $\supp(\Psi_{\alpha})\cap S^{n_X+n_Y-1}\subseteq B_{\rho}(w_i)$ for some $i$.\\

On the support of $W_{\lambda,D,\alpha,k}^z$, we have
$|D(x,y)|\approx 2^{kd_D}$ since $D$ is a homogeneous
polynomial which does not vanish away from the origin.
Here $d_D$ is the degree of $D$. We first apply Lemma \ref{operator van der Corput} to two variables $x_s$ and
$y_t$ for which $\partial_{x_s}\partial_{y_t}S\neq 0$ on the support of $\Psi_{\alpha}$, and then make use of the size estimate to other variables. This will lead to the following estimate:
\begin{eqnarray}\label{sec4 L2 damping osci est}
\|W_{\lambda,D,\alpha,k}^z\|_2
&\leq & C(1+|z|^2)
\Big(|\lambda|2^{(k-2)d}\Big)^{-1/2}
2^{kd_D\sigma}
2^{k(n_X-1)/2}
2^{k(n_Y-1)/2}  \nonumber \\
&\leq &
C(1+|z|^2)
\Big(|\lambda|2^{kd}\Big)^{-1/2}
2^{kd_D\sigma}
2^{k(n_X+n_Y)/2}
\end{eqnarray}
where $\sigma:=\Re(z)$. In view of $|x|\approx 2^k$ and $|y|\approx 2^k$ on the support of $W_{\lambda,D,\alpha,k}^z$, the Schur test gives
\begin{equation}\label{sec4 L2 damping size est}
\|W_{\lambda,D,\alpha,k}^z\|_2
\leq
C2^{kd_D\sigma}
2^{k(n_X+n_Y)/2}.
\end{equation}
For $\sigma\geq \frac{d-n_X-n_Y}{2d_D}$, the exponent of $2^k$ is nonnegative in (\ref{sec4 L2 damping osci est}). This implies $\|W_{\lambda,D,\alpha,k}^z\|_2\lesssim |\lambda|^{-1/2}$.
By Lemma \ref{sec2 almost ortho lemma}, as shown in our proof of Theorem \ref{sec3 main thm 1}, the desired damping estimate holds for $W_{\lambda,D,\alpha}^z=\sum_kW_{\lambda,D,\alpha,k}^z$. For $-\min\{n_X,n_Y\}/d_D < \sigma <\frac{d-n_X-n_Y}{2d_D}$, a convex combination of the above two estimates, annihilating the exponent of $2^k$, gives the desired estimate.\\

\textbf{Case 2} $\supp(\Psi_{\alpha})\cap S^{n_X+n_Y-1}\subseteq B_{\rho}(x_i^{\ast})$ for some $i$.\\

The oscillation and size estimates in Case 1 are still true here. However, the almost orthogonality property there does not hold now. As shown in (\ref{sec3 general damping esti}), we need only show the optimal decay $|\lambda|^{-1/2}$ for the critical exponent $\sigma=\frac{d-n_X-n_Y}{2d_D}$.
For this estimate, we claim that there exists a positive number $\delta>0$ such that
\begin{equation}\label{sec4 almost ortho in X space}
\|W_{\lambda,D,\alpha,k}^zW_{\lambda,D,\alpha,l}^{z~\ast}\|
\leq
C(z)|\lambda|^{-1} 2^{-|k-l|\delta}, ~~~\Re(z)=\frac{d-n_X-n_Y}{2d_D}.
\end{equation}
Note that $|x|\approx 2^k$ in the support of $W_{\lambda,D,\alpha,k}^z$. Hence $W_{\lambda,D,\alpha,k}^{z~\ast}
W_{\lambda,D,\alpha,l}^{z}=0$ provided that $|k-l|$ is sufficiently large.

To establish (\ref{sec4 almost ortho in X space}), we shall further impose smallness conditions on the open cover $\{\mathcal{O}_{\alpha}\}$, constructed in the proof of Theorem \ref{sec3 main thm 1}. Choose two open circular cones $U_{\alpha}$ and $V_{\alpha}$, with the same vertex at the origin, such that\\

(i) $\supp(\Psi_{\alpha})\subseteq \overline{U_{\alpha}}$ and $\overline{U_{\alpha}}\backslash \{0\} \subseteq V_{\alpha}$. Here $\overline{U_{\alpha}}$ denotes the closure of $U_{\alpha}$.\\

(ii)$\supp(\Psi_{\alpha})$, $U_{\alpha}$ and $V_{\alpha}$ are so small that for some $i_{\alpha}$,
$(w\cdot \nabla_x)\partial_{ y_{i_{\alpha}} }S(x,y)$
has fixed sign and does not vanish for all $w\in\mathbb{P}_X(\overline{U_{\alpha}}\backslash \{0\})$ and $(x,y)\in \overline{U_{\alpha}}\backslash \{0\}$. Here $\mathbb{P}_X$ is the projection from $\bR^{n_X}\times\bR^{n_Y}$ onto $\bR^{n_X}$, i.e., $\mathbb{P}_X(x,y)=x$.\\

The assumption (ii) does not lose generality since $\nabla_yS(x,0)$ is radially nondegenerate in the $X-$space. The assumption (i) implies that there exists a large number $N_{\alpha}\geq 1$ such that if $x,x'\in \mathbb{P}_X(U_{\alpha})$ and $|x|/|x'|\geq N_{\alpha}$ then $x-x'\in \mathbb{P}_X(V_{\alpha})$. Combining this observation together the assumption (ii), we obtain
\begin{equation}\label{sec4 size of mixed Hessin}
|\partial_{  y_{i_{\alpha}}  }S(x,y)-
\partial_{  y_{i_{\alpha}}  }S(u,y)|
=\Big|(x-u)\cdot \int_0^1 \nabla_x\partial_{ y_{i_{\alpha} } } S( \theta x+ (1-\theta)u, y) d\theta  \Big|
\approx
2^{k(d-1)}
\end{equation}
for all $(x,y),(u,y)\in U_{\alpha}$ and
$\Phi( x/2^k, y/2^k )
\Phi( u/2^l, y/2^l )\neq 0$, provided that $|k-l|$, assuming $k\geq l$, is sufficiently large.

In what follows, our purpose is to prove the almost orthogonality estimate (\ref{sec4 almost ortho in X space}) by the $TT^{\ast}$ method. First observe that the integral kernel associated with $W_{\lambda,D,\alpha,k}^zW_{\lambda,D,\alpha,l}^{z~\ast}$
is given by
\begin{eqnarray*}
K(x,u)
&=&
\int_{ \bR^{n_Y} } e^{i\lambda[S(x,y)-S(u,y)]}
\Psi_{\alpha}(x,y) \overline{ \Psi_{\alpha} (u,y) }
\Phi\Big( \frac{x}{2^k}, \frac{y}{2^k} \Big)
\overline{\Phi\Big( \frac{u}{2^l}, \frac{y}{2^l} \Big)}\times\\
& &~~~~~~~~\quad\quad
|D(x,y)|^z |D(u,y)|^{\overline{z}}
\varphi(x,y) \overline{\varphi(u,y)}dy.
\end{eqnarray*}
Since
$W_{\lambda,D,\alpha,k}^zW_{\lambda,D,\alpha,l}^{z~\ast}$
and $W_{\lambda,D,\alpha,l}^{z}W_{\lambda,D,\alpha,k}^{z~\ast}$
have equal $L^2$ operator norms, we can assume $k\geq l$ in the above estimate.

For $f\in C^{1}$, we define a linear differential operator
\begin{equation*}
\mathcal{D}f(y)
=
\Big[ i\lambda \Big(\partial_{ y_{ i_{\alpha} }}S(x,y)
-
\partial_{ y_{ i_{\alpha} }}S(u,y) \Big) \Big]^{-1}
\partial_{ y_{ i_{\alpha} }} f(y)
\end{equation*}
and its transpose $\mathcal{D}^t$ by the equality $\int \mathcal{D}f(y)g(y)dy=\int f(y) \mathcal{D}^tg(y) dy$
for all $f,g\in C^1_c$. It is clear that $\mathcal{D}e^{i\lambda[S(x,y)-S(u,y)]}
=e^{i\lambda[S(x,y)-S(u,y)]}$.
By integration by parts, we have
\begin{eqnarray}\label{sec4 Kernel of TTsat}
K(x,u)
&=&
\int_{ \bR^{n_Y} } e^{i\lambda[S(x,y)-S(u,y)]}
\mathcal{D}^t
\Big( \Psi_{\alpha}(x,y) \overline{ \Psi_{\alpha} (u,y) }
\Phi\Big( \frac{x}{2^k}, \frac{y}{2^k} \Big)
\overline{\Phi\Big( \frac{u}{2^l}, \frac{y}{2^l} \Big)}\times\nonumber \\
& &~~~~~~~~\quad\quad
|D(x,y)|^z |D(u,y)|^{\overline{z}}
\varphi(x,y) \overline{\varphi(u,y)} \Big)dy.
\end{eqnarray}

The phase function $S(x,y)-S(u,y)$ can be viewed as a polynomial in $y_{ i_{\alpha} }$ of degree $\leq d$ with other variables fixed. Without loss of generality, we assume $i_{\alpha}=1$. For arbitrary $x,u\in\bR^{n_X}$, let $E_{\alpha,k,l}(x,u)$ be the set of $y\in \bR^{n_Y}$ such that the integrand in (\ref{sec4 Kernel of TTsat}) does not vanish. Take an arbitrary point $y=(y_1,y')\in E_{\alpha,k,l}(x,u)$ with $y'=(y_2,\cdots,y_{n_Y})\in\bR^{n_Y-1}$. If $E_{\alpha,k,l}(x,u)$ is nonempty, then the above assumptions (i) and (ii) imply that as a function of $y_1$, with $y'$ fixed, $\partial_{y_1}S(x,y)-\partial_{y_1}S(u,y)$ does not change sign and its absolute value $\approx 2^{k(d-1)}$ on an interval $I$ with length $|I|\approx 2^l$.
By Lemma \ref{sec2 poly type func lemma}, we have
\begin{eqnarray*}
\sup_{y_1\in I}
\Big|\partial_{y_1}
\Big(
\partial_{y_1}S(x,y)-\partial_{y_1}S(u,y)
\Big)^{-1}\Big|
&\lesssim&
|I|^{-1}2^{-k(d-1)}
\approx 2^{-l}2^{-k(d-1)},\\
\sup_{y_1\in I}
\Big|\partial_{y_1}|D(x,y)|^z\Big|
&\lesssim&
|I|^{-1}2^{kd_D\sigma}
\approx 2^{-l}2^{kd_D\sigma},\\
\sup_{y_1\in I}
\Big|\partial_{y_1}|D(u,y)|^z\Big|
&\lesssim&
|I|^{-1}2^{kd_D\sigma}
\approx 2^{-l}2^{ld_D\sigma},\\
\end{eqnarray*}
where $\sigma=\Re(z)$ and the above implicit constants are independent of $x$, $u$ and $y'$.

Recall that we have assumed $k\geq l$. The $y_1-$partial derivative of other cut-off functions in (\ref{sec4 Kernel of TTsat}) is bounded by a constant multiple of $2^{-l}$. Hence we deduce the following pointwise estimate from (\ref{sec4 Kernel of TTsat}):
\begin{equation*}
|K(x,y)|
\lesssim
\Big( |\lambda| 2^{k(d-1)} \Big)^{-1} 2^{kd_D\sigma}
2^{ld_D\sigma} 2^{l(n_Y-1)} \chi_{ \{|x|\approx 2^k\} }(x)
\chi_{ \{|u|\approx 2^l\}  }(u).
\end{equation*}
Then
\begin{eqnarray*}
\sup_{x}\int_{\bR^{n_X}}|K(x,u)|du
& \lesssim &
\Big( |\lambda| 2^{k(d-1)} \Big)^{-1} 2^{kd_D\sigma}
2^{ld_D\sigma} 2^{l(n_Y-1)} 2^{ln_X},\\
\sup_{u}\int_{\bR^{n_X}}|K(x,u)|dx
& \lesssim &
\Big( |\lambda| 2^{k(d-1)} \Big)^{-1} 2^{kd_D\sigma}
2^{ld_D\sigma} 2^{l(n_Y-1)} 2^{kn_X}.
\end{eqnarray*}
By the Schur test, we obtain
\begin{eqnarray*}
\|W_{\lambda,D,\alpha,k}^z
W_{\lambda,D,\alpha,l}^{z~\ast}\|_2
&\lesssim&
\Big( |\lambda| 2^{k(d-1)} \Big)^{-1} 2^{kd_D\sigma}
2^{ld_D\sigma} 2^{l(n_Y-1)} 2^{(k+l)n_X/2}\\
&\lesssim&
|\lambda|^{-1}2^{-|k-l|\delta}, ~~~\delta=\frac{d+n_Y}{2}-1>0,
\end{eqnarray*}
where $\sigma=\Re(z)$ is given by (\ref{sec4 almost ortho in X space}). By the Cotlar-Knapp-Stein almost orthogonality principle (see Stein \cite{stein}), we obtain $\|W_{\lambda,D,\alpha}^z\|_2\lesssim |\lambda|^{-1/2}$. \\

\textbf{Case 3} $\supp(\Psi_{\alpha})\cap S^{n_X+n_Y-1}\subseteq B_{\rho}(z_j^{\ast})$ for some $j$.\\

In this case, the damping estimate with critical damping exponent can be proved as in Case 2, with the roles of $x$ and $y$ interchanged. The details are omitted here.\\

\textbf{Step 2}. Proof of the endpoint $L^p$ estimates.\\

In Case 1, $W_{\lambda,D,\alpha}^z$ satisfies (i) $\|W_{\lambda,D,\alpha}^z\|_2\lesssim |\lambda|^{-1/2}$ with $\Re(z)=\frac{d-n_X-n_Y}{2d_D}$ and
(ii) $\|W_{\lambda,D,\alpha}^z\|_1\lesssim 1$ with
$\Re(z)=-\frac{n_X}{d_D}$. By interpolation, we obtain the left endpoint $L^p$ estimate (\ref{sec3 main Lp esti}).

In Case 2, a slight modification is needed in our argument. We shall replace $D(x,y)$ by $|x|^{d_D}$ in the definition of $W_{\lambda,D,\alpha}^z$. Then $\|W_{\lambda,D,\alpha}^z\|_2\lesssim |\lambda|^{-1/2}$ still holds for $\Re(z)=\frac{d-n_X-n_Y}{2d_D}$. The reason is that $D(x,y)$ and $|x|^{d_D}$, together with their partial derivatives, have the same upper bounds in our proof of this critical $L^2$ damping estimate. On the other hand, note that $\|W_{\lambda,D,\alpha}^zf\|_{L^{1,\infty}}\lesssim \|f\|_{L^1}$ for $\Re(z)=-\frac{n_X}{d_D}$. By Lemma \ref{sec2 interp change of meas}, the desired endpoint $L^p$ estimate follows.

The Case 3 can be treated in the same way as Case 2. However, we do not need change the damping factor $D$. Now the critical $L^2$ damping estimate in Case 2 is still true. For $\Re(z)=-\frac{n_X}{d_D}$, the stronger estimate $\|W_{\lambda,D,\alpha}^z\|_{1}\lesssim 1$ holds, as in our proof of Theorem \ref{sec3 main thm 1}. By a duality argument, we are able to prove the right endpoint $L^p$ estimate for (\ref{sec3 main Lp esti}). Thus the proof of the theorem is complete.
\end{proof}

\noindent{\bf Acknowledgements.} We would like to thank Shaozhen Xu for explanation of his work and sharing useful ideas with us.

\end{document}